\documentclass[a4paper,12pt]{amsart}

  \usepackage{amssymb,amsthm}
\usepackage{graphicx,color}    
 \usepackage{changes,soul}    

\usepackage{url} 

\newtheorem{theorem}{Theorem}[section]
\newtheorem{corollary}[theorem]{Corollary}
\newtheorem{proposition}[theorem]{Proposition}

\theoremstyle{definition}
\newtheorem{definition}[theorem]{Definition}

\newtheorem{remark}[theorem]{Remark}

\def\r{\mathbb R}
\def\h{\mathbb H}
\def\l{\mathbb L}
\def\s{\mathbb S}
\def\n{\mathbf n}
\def\d{\mathsf d}
\begin{document}

\title{Extension of a problem of Euler in $\h^2$ and in $\s^2$}
 
\author{Muhittin Evren Aydin}
\address{Department of Mathematics, Faculty of Science, Firat University, Elazig,  23200 Turkey}
\email{meaydin@firat.edu.tr}
\author{Antonio Bueno}
\address{Departamento de Matem\'aticas\\  Universidad de Murcia,  30100 Murcia, Spain}
\email{jabueno@um.es}
\author{Rafael L\'opez}
\address{Departamento de Geometr\'{\i}a y Topolog\'{\i}a Universidad de Granada 18071 Granada, Spain}
\email{rcamino@ugr.es}
\subjclass{53A04, 49K05, 74G65}
\keywords{moment of inertia, stationary curve, sphere, hyperbolic plane}
\begin{abstract}
In this paper, we extend the notion of stationary curves  with respect to the moment of inertia from a point $N$ in the Euclidean plane $\r^2$ to the case that the ambient space is either the hyperbolic plane $\h^2$ or the sphere $\s^2$.   We characterize the critical points of this energy in terms of the curvature of the curve and the distance to $N$.  In $\h^2$, we prove that the only closed stationary curves are circles centered at $N$. In $\s^2$, we estimate the value of $\alpha$ for closed curves according to the hemisphere of $\s^2$ in which the curve lies. In addition, we find the first integrals  of the ODEs that describe the parametrizations of stationary curves in both ambient spaces. Finally, we consider the energy minimization problem for curves connecting two points collinear with $N$, in particular solving the case of geodesics.
\end{abstract}

\maketitle
\section{The statement of problem}

 The purpose of this paper is to extend to the hyperbolic plane $\h^2$ and the sphere $\s^2$ the following problem investigated by Euler for planar curves: to find planar curves $\gamma\colon I\subset\r \to\r^2$ which   minimize the   energy 
$$
E_\alpha[\gamma]=\int_\gamma |\gamma(s)|^\alpha\, ds,
$$
where $\alpha \in\r$ is a parameter,  and $s$ is the   arc-length parameter (see \cite{eu}) . Critical points of the energy $E_\alpha$ are curves characterized by the equation 
\begin{equation}\label{eq0}
\kappa=\alpha\frac{\langle \n,\gamma\rangle}{|\gamma|^2},
\end{equation}
where $\kappa$ and $\n$ are the curvature and the unit normal vector of $\gamma$, respectively. Euler solved \eqref{eq0}, obtaining  explicit parametrizations of these curves. On the other hand, Mason  found the minimizers of the energy $E_2$ for all curves joining two fixed points of $\r^2$ \cite{ma}; see also \cite{ca,to}. Recently, Dierkes and the third author have obtained  a general approach to find the minimizers of $E_\alpha$  for all values of $\alpha$  \cite{dl1}. 

It is natural to extend this problem to the other geometries with constant curvature, namely, the hyperbolic plane $\h^2$ (negative curvature) and the sphere $\s^2$ (positive curvature). As in the Euclidean plane $\r^2$,   we     fix the reference point with respect to which  the intrinsic distance of the space is measured. For this, we use the following models for $\h^2$ and $\s^2$.  For the hyperbolic plane $\h^2$, we will consider the Lorentzian model and for the sphere $\s^2$, we see the sphere as the set of points of $\r^3$ equidistant from the origin. In both spaces, we take the north pole $N=(0,0,1)$ as the reference point. 

Let  $\gamma\colon I\to \h^2$ or $\s^2$ be  a curve parametrized by arc-length, denoted $\gamma=\gamma(s)$. Given $\alpha\in\r$, define the energy 
\begin{equation}\label{eq1}
E_\alpha[\gamma]=\int_\gamma \d(s)^\alpha \, ds, 
\end{equation}
where $\d(s)=\mbox{dist}(\gamma(s),N)$. If $\alpha=2$, then the energy represents the moment of inertia of   $\gamma$ with respect to the point $N$. In the general case for $\alpha$, we can interpret the energy $E_\alpha[\gamma]$ as a power of the moment of inertia with respect to $N$. If $\alpha=0$, then $E_\alpha$ is simply the length of the curve, and the corresponding critical points of the energy are geodesics in the space. From now on, we will discard this case. 

\begin{definition}
A curve $\gamma$ in $\h^2$ or $\s^2$ is called an $\alpha$-stationary curve if $\gamma$ is  a critical point of the energy $E_\alpha$.
\end{definition}

When it is not emphasized the value of $\alpha$, we simply say  stationary curves. The characterization of the stationary curves will be obtained by the Euler-Lagrange equations of the energy $E_\alpha$. For that, we will require along this paper that $\gamma$ does not cross the point $N$ in order to have differentiability of $E_\alpha$. 

Let $(x,y,z)$ be the canonical coordinates in $\r^3$. We denote by $\langle ,\rangle_\epsilon =(dx)^2+(dy)^2+\epsilon (dz)^2$ the Euclidean or Lorentzian metric according to whether $\epsilon=1$ or $\epsilon=-1$, respectively. For $\epsilon=1$, we simply write $\langle ,\rangle_\epsilon =\langle ,\rangle$. We will prove that a curve $\gamma$ in $\h^2$ or in $\s^2$ is an $\alpha$-stationary curve if and only if its curvature $\kappa$ satisfies
\begin{equation}\label{k0}
\kappa=\alpha\frac{\langle \n,\xi\rangle_\epsilon}{\d },
\end{equation}
where $\n$ denotes the unit normal vector  of $\gamma$ and $\xi$   is the unitary tangent vector  at $\gamma(s)$ of the geodesic joining $N$ with  $\gamma(s)$. 

If we compare Equation \eqref{k0} with its analogue \eqref{eq0} in the Euclidean plane $\r^2$, we arrive at an interesting conclusion about   the exponent $2$ in the denominator on the right hand-side of \eqref{eq0}, namely $|\gamma|^2$. In $\r^2$, the ray from the origin $0\in\r^2$ to the point $\gamma(s)$ is parametrized by $t\mapsto t\gamma(s)$, $t\geq 0$. At $\gamma(s)$, the unit tangent vector  to this ray is $\frac{\gamma(s)}{|\gamma(s)|}$. Denoting $\xi=\frac{\gamma}{|\gamma|}$, Equation \eqref{eq0} can be now expressed by 
\begin{equation}\label{eq00}
\kappa=\alpha\frac{\langle\n,\gamma\rangle}{|\gamma|^2}=\alpha\frac{\langle \n,\xi\rangle}{|\gamma|}=\alpha\frac{\langle \n,\xi\rangle}{\d},
\end{equation}
where $\d$ is the Euclidean distance from the origin $0$ to $\gamma(s)$.
This provides a new perspective on the classical characterization of stationary curves in $\r^2$ because the exponent $2$ of $|\gamma|$ in the denominator reduces to $1$. In addition, and as another interesting consequence, Equations \eqref{eq0} and \eqref{eq00} have the same form, which it is expectable independently of the space form.

Since we work in two different space forms, $\h^2$ and $\s^2$, we divide the investigation into two sections. As the arguments are similar in both spaces, we omit the details for $\s^2$. After deriving the Euler–Lagrange equations for the energy \eqref{eq1}, we provide the characterization \eqref{k0} of stationary curves (Propositions \ref{h-p1} and \ref{s-p1}). Such curves with constant curvature are determined in Theorem \ref{h-ccurv} and Proposition \ref{s-ccurv}. Finally, in Theorem \ref{t29} and \ref{t36}, we obtain the first integrals of the ODEs describing the parametrizations of stationary curves. 

In comparing the ambient spaces $\h^2$ and $\s^2$, two main differences appear: one related to closed stationary curves, and the other to energy minimization.
\begin{enumerate}
\item We prove that in $\h^2$, the only stationary curves are circles centered at $N$ (see Theorem \ref{t1}). In contrast to $\h^2$, the space $\s^2$ is compact. For closed curves in $\s^2$, we also determine the value of $\alpha$ according to the hemisphere of $\s^2$ in which the curve lies (see Theorem \ref{t12}).

\item Given two points $p_1$ and $p_2$ lying on the same geodesic with $N$, Theorems \ref{th1} and \ref{ts1} state that such geodesics in $\h^2$ or $\s^2$ are minimizers of the energy \eqref{eq1}. However, in $\s^2$ we exclude the case that $p_1$ and $p_2$ lie on the different rays and  the minimizing geodesic from $p_1$ to $p_2$ does not pass through $N$ because of the lack of  a suitable estimate involving the geodesic and $N$.
\end{enumerate}
 
\section{Stationary curves in hyperbolic plane} \label{sec-h}

In what follows, we investigate in $\h^2$, through separate subsections, the Euler–Lagrange equation for \eqref{eq1}, examples of stationary curves, applications of the maximum principle, parametrizations of such curves and, finally minimization of energy.

\subsection{The Euler-Lagrange equation}

Let $\l^3=(\r^3,\langle,\rangle_\epsilon)$ be the Lorentz-Minkowski space. The Lorentzian model of the hyperbolic plane $\h^2$ is the surface 
$$\h^2=\{(x,y,z)\in \l^3 \colon x^2+y^2-z^2=-1, z>0\}$$ endowed with the induced metric from $\l^3$. A parametrization of $\h^2$ is given by 
$$\Psi(u,v)=(\sinh (u)\cos (v),\sinh (u)\sin (v),\cosh (u)),\quad (u,v)\in\r^2.$$
We take the north pole $N=(0,0,1)$ as a reference point for the distance function. Let $\gamma\colon [a,b]\to\h^2$ be a curve, $\gamma=\gamma(t)$, given by \begin{equation}  \label{par-h}
\gamma(t)=\Psi(u(t),v(t))=(\sinh (u(t))\cos (v(t)),\sinh (u(t))\sin (v(t)),\cosh (u(t))),
\end{equation}
for some functions $u=u(t)$, $v=v(t)$ on $[a,b]$. Letting $|\gamma'|_\epsilon=\sqrt{|\langle \gamma',\gamma'\rangle_\epsilon|}$, the line element is $\sqrt{u'^2+\sinh^2 (u) v'^2}\,  dt
$ and the distance from $\gamma(t)$ to $N$ is $\d(t)=u(t)$. Then the expression of the energy \eqref{eq1} in coordinates $(u,v)$ is 
\begin{equation} \label{en-h}
E_\alpha[\gamma]=\int_a^b u^\alpha\sqrt{u'^2+\sinh^2 (u) v'^2}\, dt.
\end{equation}

We now derive the characterization \eqref{k0} of the critical points of $E_\alpha$   by means of the Euler-Lagrange equations corresponding to the functional \eqref{en-h}. First we determine the unit normal vector $\n(t)$ of $\gamma(t)$, which is given by
\begin{equation}\label{nn}
\begin{split}
\n(t)&=\frac{\gamma'(t)\times_\epsilon\gamma(t)}{|\gamma'(t)|_\epsilon} \\
&=\frac{1}{\sqrt{u'^2+\sinh^2 (u) v'^2}}
\begin{pmatrix}
u' \sin (v) + \sinh (u) \cosh (u) v' \cos (v) \\[0.6em]
\sinh (u) \cosh (u) v' \sin (v) - u' \cos (v) \\[0.6em]
(\sinh (u))^2 v'
\end{pmatrix},
\end{split}
\end{equation}
where $\times_\epsilon$ is the cross product in $\l^3$. On the other hand, the geodesic curvature $\kappa$ of $\gamma(t)$ is 
\begin{equation}\label{k1}
\begin{split}
\kappa(t)&=\frac{\langle\gamma''(t),\n(t)\rangle_\epsilon}{|\gamma'(t)|_\epsilon^2}\\
&=\frac{1}{|\gamma'(t)|_\epsilon^3}\left(v'\cosh (u)(2u'^2+v'^2\sinh^2(u))+\sinh (u)(u'v''-u''v')\right).
\end{split}
\end{equation}

Next, we set the integrand of \eqref{en-h} as
$$
J=J(u,u',v')=u^\alpha\sqrt{u'^2+\sinh^2 (u) v'^2} .
$$
The Euler-Lagrange equations are obtained by computing
\begin{equation} \label{el-h0} 
\frac{ \partial J}{\partial u}=\frac{d}{dt}\left(\frac{\partial J}{\partial u'}\right), \quad  \frac{ \partial J}{\partial v}=\frac{d}{dt}\left(\frac{\partial J}{\partial v'}\right).
\end{equation}
Notice that in the second equation, we have $\frac{\partial J}{\partial v}=0$. A computation of the two equations in   \eqref{el-h0}  gives 
\begin{equation*} 
\begin{split}
0&=u^{\alpha-1}v'\sinh (u)\Big(\alpha\sinh(u) v'(u'^2+(\sinh (u))^2 v'^2)\\
&+u\left(\sinh u(u'v''-u''v') +v'\cosh (u)(2u'^2+v'^2(\sinh (u))^2)\right)\Big),\\
0&=u^{\alpha-1}u'\sinh (u)\Big(\alpha\sinh(u) v'(u'^2+(\sinh (u))^2 v'^2)\\
&+u\left(\sinh u(u'v''-u''v') +v'\cosh (u)(2u'^2+v'^2(\sinh (u))^2)\right)\Big).
\end{split}
\end{equation*}
By the expression of $\kappa$ in \eqref{k1}, both equations write as
$$0=u^{\alpha-1}v'\sinh (u)\Big(\alpha\sinh(u) v'|\gamma'|_\epsilon^2+u\kappa|\gamma'|_\epsilon^3\Big),$$
$$0=u^{\alpha-1}u'\sinh (u)\Big(\alpha\sinh(u) v'|\gamma'|_\epsilon^2+u\kappa|\gamma'|_\epsilon^3\Big).$$
Since $\gamma$ is regular, then $u'\not=0$ or $v'\not=0$. Thus we deduce that the parenthesis in the above equations must vanish, that is, 
$$\alpha\sinh(u) v'|\gamma'|_\epsilon^2+u\kappa|\gamma'|_\epsilon^3=0.$$
As a conclusion, we have the following   characterization of the stationary curves.

\begin{proposition}\label{h-p}
Let $\gamma(t)$ be a curve in $\h^2$  parametrized by \eqref{par-h}. Then $\gamma$ is an $\alpha$-stationary curve if and only if  
\begin{equation} \label{el-h22}
\kappa=-\alpha\frac{v'\sinh u }{u|\gamma'|_\epsilon}.
\end{equation}
\end{proposition}

We now establish an expression of \eqref{el-h22} which allows to compare   with its analogue \eqref{eq0} in the Euclidean plane. 

Consider all geodesics from $N$ parametrized by arc-length. We say that these geodesics  are {\it rays} from $N$. For any point $p\in\h^2$, there is a unique ray from $N$ passing through $p$. Denote by
$$\xi=\xi(p)$$
the tangent vector of this ray at $p$. We next obtain an expression of $\xi(p)$ in the coordinates $\Psi=\Psi(u,v)$.   The rays from $N$ are given by the intersections of $\h^2$ with the planes containing the $z$-axis, and are given by
$$
u\mapsto \Psi(u,v_0), \quad u>0,
$$
for all $v_0\in \r$. Notice that $|\Psi_u|_\epsilon=1$. Thus   
$$\xi(\gamma(t))=\Psi_u(u(t),v(t))= (\cosh(u(t))\cos(v),\cosh(u(t))\sin(v),\sinh(u(t))).$$
From the expression of $\n(t)$ in \eqref{nn}, we have
$$
\langle \n,\xi\rangle_\epsilon=-\frac{v'\sinh (u)}{|\gamma'|_\epsilon}.
$$
Comparing this identity with Proposition \ref{h-p}, we arrive at the final characterization of $\alpha$-stationary curves in $\h^2$.

\begin{proposition}\label{h-p1}
Let $\gamma(t)$ be a curve in $\h^2$  parametrized by \eqref{par-h}. Then $\gamma$ is an $\alpha$-stationary curve if and only if  
\begin{equation} \label{el-h2}
\kappa=\alpha\frac{\langle \n,\xi\rangle_\epsilon}{\d }.
\end{equation}
Here, $\n$ denotes the unit normal vector of $\gamma$, $\d$ is the distance from $N$, and $\xi$ is the unitary tangent vector of the ray at $\gamma(t)$ which joins $N$ with $\gamma$.  
\end{proposition}

\begin{remark}\label{h-rem}
Any isometry that preserves the solutions of \eqref{el-h2} requires to fix $N$. Hence, rotations about the $z$-axis and reflections about planes containing the $z$-axis preserve the solutions of \eqref{el-h2}.
\end{remark}

\subsection{Examples of stationary curves}

We show some examples of stationary curves in the family of circles and geodesics.
\begin{enumerate}
\item Geodesics through $N$ are $\alpha$-stationary curves for all $\alpha$. A such a geodesic is parametrized by $\gamma(t)=\Psi(t,v_0)$. Then $\n(t)=(\sin(v_0),-\cos(v_0),0)$ and $\xi(t)=(\cos (v_0) \cosh (u(t)),\sin (v_0) \cosh (u(t)),\sinh (u(t)))$. Thus $\langle \n(t),\xi(t)\rangle_\epsilon=0$. On the other hand, we know $\kappa=0$, hence \eqref{el-h2} holds for all $\alpha$.
\item Circles centered at $N$. If $r$ is the radius of the circle, we see that the circle an $\alpha$-stationary curve for $\alpha=r\cosh(r)$. A parametrization of the circle is 
$$\gamma(t)=\Psi(r,t)=(\sinh(r)\cos(t),\sinh(r)\sin(t),\cosh(r)).$$
Consider the   normal vector 
$$\n(t)=-(\cosh(r)\cos(t),\cosh(r)\sin(t),\sinh(r)).$$
Then, $\xi(t)=-\n(t)$ and  $\langle\n(t),\xi(t)\rangle_\epsilon=-1$. The curvature is given by 
$$\kappa=\frac{\langle\gamma''(t),\n(t)\rangle_\epsilon}{|\gamma'(t)|^2_\epsilon}=\coth(r).$$
Since $\d=r$, then \eqref{el-h2} holds for $\alpha=-r\coth(r)$.
\end{enumerate}

Recall that straight-lines in $\r^2$ passing through the origin are $\alpha$-stationary curves for all $\alpha$. These curves are analog to the example given in item (1). But, item (2) has a significant distinct: in $\r^2$, every circle of radius $r$ centered at the origin is an $\alpha$-stationary curve with $\alpha=-1$, independently of the value of $r$. This contrasts to with the circles of $\h^2$ at centered $N$ because the value of $\alpha$ varies with the radius. 

Motivated by these examples, we will find all stationary curves in $\h^2$ which have constant curvature. Notice that in $\h^2$, besides geodesics ($\kappa=0$) and circles ($\kappa>1$), there are more curves with constant curvature, namely, equidistant lines ($0<\kappa<1$) and horocycles ($\kappa=1$).   For the next computations, we need to have the description of the curves of $\h^2$ with constant curvature.  

\begin{proposition} \label{p-tau-h} 
The curves in $\h^2$ with constant curvature are described as follows. Let $\delta\in \{-1,0,1\},$ and let $\tau\in\r$. Then any curve with constant curvature is given by
$$C_{a,\tau}=\{p\in\h^2\colon\langle p,a\rangle_\epsilon=\tau\},$$
where $\delta=\langle a,a\rangle_\epsilon$. The normal is
\begin{equation}\label{n4}
\n(p)=-\lambda(\tau p+a),\quad\lambda=\frac{1}{\sqrt{\tau^2+\delta}},
\end{equation}
and the curvature is $\kappa=\lambda\tau$. The types are the following:
\begin{enumerate}
\item Geodesics. Here $\delta=1$ and $\tau=0$. We have $\kappa=0$.
\item Equidistant line. Here $\delta=1$ and $\tau\not=0$. Now $\kappa=\frac{\tau}{\sqrt{\tau^2+1}}\in (-1,0)\cup (0,1)$.
\item Horocycles. Here $\delta=0$ and $\tau\not=0$. The curvature is $\kappa=1$.
\item Circles. Here $\delta=-1$ and $|\tau|>1$. The curvature is $\kappa=\frac{\tau}{\sqrt{\tau^2-1}}$.
\end{enumerate}
\end{proposition}

The classification of the   stationary curves in $\h^2$ of constant curvature is given in the following result.

\begin{theorem} \label{h-ccurv}
The only $\alpha$-stationary curves in $\h^2$ with constant curvature are:
\begin{enumerate}
\item Geodesics passing through $N$. This holds for all value of $\alpha$.
\item Circles of radius $r$ centered at $N$. This holds for $\alpha=-r\coth (r)$ and all $r>0$. 
\end{enumerate}
\end{theorem}

\begin{proof}
Let $C_{a,\tau}$ be a curve with constant curvature. Let  $a=(a_1,a_2,a_3) \in \l^3$ and $p\in C_{a,\tau}$ given by  $p=\Psi(u(t),v(t))$. Since $\langle p,a\rangle_\epsilon=\tau$, we have 
\begin{equation} \label{tau-ah0}
\tau=a_1\sinh(u(t))\cos(v(t))+a_2\sinh(u(t))\sin(v(t))-a_3\cosh(u(t)).
\end{equation}
Moreover, we have
$$
\langle a,\xi \rangle_\epsilon=a_1\cosh(u(t))\cos(v(t))+a_2\cosh(u(t))\sin(v(t))-a_3\sinh(u(t)).
$$
Then, it follows that 
\begin{equation}\label{n5}
\langle a,\xi \rangle_\epsilon =\frac{\tau\cosh(u(t))+a_3}{\sinh(u)}.
\end{equation}
Since $\langle p, \xi\rangle_\epsilon=0$, then from the expression of $\n$ in \eqref{n4} we have 
 $\langle\n,\xi\rangle_\epsilon=-\lambda\langle a,\xi\rangle_\epsilon$. From \eqref{n5}, and because $\kappa=\lambda\tau$, Equation \eqref{el-h2} writes as
$$\tau=-\alpha\frac{\tau\cosh(u(t))+a_3}{u(t)\sinh(u(t))},$$
or equivalently, 
\begin{equation} \label{tau-ah1}
\alpha\tau\cosh(u(t))+\tau u(t)\sinh(u(t))+\alpha a_3=0.
\end{equation}
We distinguish two cases:
\begin{enumerate}
\item Case $\tau=0$. In this case, Equation \eqref{tau-ah1} implies $a_3=0$ and so by \eqref{tau-ah0} it must be $v(t)=v_0\in\r$. Consequently, these geodesics cross  the point $N$, and are $\alpha$-stationary curves, for every value of $\alpha$.

\item Case $\tau\neq 0$. If $u(t)$ is not a constant function, and because $u(t)>0$,  then the linearly independence of the set $\{\cosh(u(t)),u(t)\sinh(u(t))\}$ yields a contradiction in \eqref{tau-ah1}. Hence, $u(t)$ must become constant, say $u=r>0$.  From  \eqref{tau-ah0}, we conclude $a_1=a_2=0$ and that $C_{a,\tau}$ is a circle centered at $N$.
\end{enumerate}
\end{proof}

Comparing with the Euclidean plane $\r^2$, besides straight-lines crossing the origin and circles centered at the origin, there are also other stationary curves with constant curvature. These curves are   circles crossing the origin. In such a case, $\alpha=-2$. However, in $\h^2$, the curves with constant curvature which cross the point $N$  satisfy   $\langle N,a\rangle_\epsilon=\tau$ and $a_3=-\tau$. Therefore, in each item of Proposition \ref{p-tau-h} there exist curves that cross the point $N$. However, except for geodesics, none of these curves is $\alpha$-stationary. 

\subsection{The maximum principle}

In this section, we will study those stationary curves that are closed. For this, we will compare the curvature of these curves with that of circles obtained in Theorem \ref{h-ccurv}. In the arguments, it is better to replace  the curvature by the comparison of the weighted curvature in the sense of manifolds with density (\cite[Sects. 3 and 8]{gr}). For a general density   $e^\phi$, where $\phi$ is a smooth function in $\h^2$, the {\it weighted curvature} $\kappa^\phi$ of $\gamma$ is defined by 
\begin{equation}\label{eq11}
\kappa^\phi=\kappa-\langle(\nabla\phi)\circ \gamma,\n\rangle_\epsilon,
\end{equation}
where $\nabla$ is the gradient operator on $\h^2$. We now particularize for $\phi(p)=\alpha \log(\d(p))$. If $X$ is a tangent vector of $\h^2$ at $p=\Psi(u,v)$, then 
$$\langle\nabla\phi,X\rangle_\epsilon=\alpha\frac{\langle \nabla \d,X\rangle_\epsilon}{\d}.$$
Now we use the parametrization $\Psi$ of $\h^2$. If $X=\frac{d}{dt}\Psi(u(t),v(t))$, then
$$\langle\nabla \d,X\rangle_\epsilon=\langle\nabla \d,\frac{d}{dt}\Psi(u(t),v(t))\rangle_\epsilon=\frac{d}{dt} \d\circ\Psi(u(t),v(t))=u'(t).$$
But  we also have 
$$u'(t)= \langle X,\Psi_u\rangle_\epsilon=\langle X,\xi\rangle_\epsilon.$$
Thus 
$$\langle\nabla\phi,X\rangle_\epsilon=\alpha\frac{\langle X,\xi\rangle_\epsilon}{\d}.$$
Once we have obtained the expression of $\nabla\phi$, the weighted curvature $\kappa^\phi$ in  \eqref{eq11}  becomes
\begin{equation}\label{kk}
\kappa^\phi=\kappa-\alpha\frac{\langle\n,\xi\rangle_\epsilon}{\d}.
\end{equation}

The maximum principle for the weighted curvature $\kappa^\phi$ is applied to obtain the following result (\cite{gr} and also in \cite[Ch. 3]{amr}).

\begin{proposition}[maximum principle]
 Let $\gamma_1$ and $\gamma_2$ be two curves in $\l^3$ tangent at $s_0$ such that $\n_1(s_0)=\n_2(s_0)$. If $\gamma_1\geq\gamma_2$ around $s_0$ with respect to the orientation $\n_i(s_0)$, then $\kappa_1^\phi(s_0)\geq \kappa_2^\phi(s_0)$. If, in addition,  $\kappa_1^\phi$ and $\kappa_2^\phi $ are constant with $\kappa_1^\phi=\kappa_2^\phi$, then $\gamma_1$ and $\gamma_2$ coincide in an open set around $s_0$. 
\end{proposition}

The maximum principle allows to characterize the class of closed  $\alpha$-stationary curves. 

\begin{theorem}\label{t1}
 The only $\alpha$-stationary closed curves in $\h^2$ are circles centered at $N$.  
\end{theorem}
\begin{proof} Let $\gamma$ be a closed $\alpha$-stationary curve. Let $C_{r}$ denote a circle centered at $N$ of radius $r>0$ and let $D_r\subset\r^2$ be the closed disk bounded by $C_r$ which contains $N$ in its interior. Let $r>0$  be sufficiently big so $\gamma(I)\subset D_{r}$. Let decrease $r$, with $r\searrow 0$, until the first intersection with $\gamma$. Suppose that this occurs for $r=r_2$. At the intersection point between $C_{r_2}$ and $\gamma$, consider   on $C_{r_2}$ the inward orientation, that is, the orientation pointing towards $D_{r_2}$. Orient $\gamma$ so that  it coincides with $C_{r_2}$ at the contact point. Thus we have $\gamma\geq C_{r_2}$. For the weighted curvature $\kappa^\phi$ corresponding to the value $\alpha$, the maximum principle yields   $\kappa^\phi_\gamma\geq  \kappa^\phi_{C_{r_2}}$. For  $\gamma$, we have  $\kappa^\phi_\gamma=0$ because $\gamma$ is an $\alpha$-stationary curve, independently from the orientation on $\gamma$. For $C_{r_2}$, and taking into account that the normal vector of $C_{r_2}$ is $-\xi$, we obtain 
$$\kappa^\phi_{C_{r_2}}=\coth(r_2)+ \frac{\alpha}{r_2}=\frac{\alpha+r_2\coth(r_2)}{r_2}.$$   Thus, by the maximum principle,
$$0\geq  \frac{\alpha+r_2\coth(r_2)}{r_2}.$$
This implies $\alpha\leq -r_2\coth(r_2)$. In other words, $\alpha\leq  \alpha_2:=-r_2\coth(r_2)$.

Analogously, we do a similar argument by taking circles $C_r$ with $r$  small. When $r$ is close to $0$, the curve $\gamma$ lies outside the domain $D_r$. We then increase the radius $r$ of $C_r$ until the first contact with $\gamma$ at $r=r_1$. Notice that $r_1\leq r_2$. At this point, we consider the outward orientation on $C_{r_1}$. Again, we have $\gamma\geq C_{r_1}$ around the contact point. By the maximum principle, it follows $0\geq \kappa^\phi_{C_{r_1}}$. Now we have
$$\kappa^\phi_{C_{r_1}}=-\coth(r_2)- \frac{\alpha}{r_2}=-\frac{\alpha+r_1\coth(r_1)}{r_1}$$
because the normal vector on $C_{r_1}$ coincides with $\xi$. The maximum principles yields
$$0\geq -\frac{\alpha+r_1\coth(r_1)}{r_1},$$
which gives $\alpha\geq -r_1\coth(r_1)$, or equivalently, $\alpha\geq\alpha_1:=r_1\coth(r_1)$. Definitively, we have proved  
$$-r_1\coth(r_1)\leq\alpha\leq-r_2\coth(r_2).$$ 
However, the function $x\mapsto -x\coth(x)$ is decreasing for $x>0$. Since $r_1\leq r_2$, we deduce $r_1=r_2$ and thus $\gamma=C_{r_1}$ as we want to prove.
\end{proof}

As a consequence of the proof, we deduce that  for some values of $\alpha$, any $\alpha$-stationary curve tends to infinity.

\begin{corollary} Let $\gamma$ be an $\alpha$-stationary curve properly immersed in $\h^2$. If $\alpha\geq -1$, then $\gamma(I)$ is not bounded.
\end{corollary}
\begin{proof}
If $\gamma(I)$ is bounded, then there is $r>0$ such that $\gamma(I)$ is contained in the domain $D_r$ bounded by a circle centered at $N$ of radius $r$. Since $\gamma$ is properly immersed, by letting $r\nearrow 0$, we arrive until the first radius $r_2$ such that $C_{r_2}$ touches $\gamma$ at some point. The maximum principle implies $\alpha\leq -r_2\coth(r_2)$. Since $-r_2\coth(r_2)<-1$, we get a contradiction. 
\end{proof} 

\subsection{Parametrizations of stationary curves}

In this subsection, we provide the explicit equations of $\alpha$-stationary curves in $\h^2$. 

\begin{theorem}\label{t29}
For the $\alpha$-stationary curves $\gamma(t)=\Psi(u(t),v(t))$ in $\h^2$ which are not of constant curvature we have 
\begin{equation} \label{classif-h}
\begin{aligned}
t(u)&=\pm\int^u\frac{s^\alpha\sinh(s)}{\sqrt{s^{2\alpha}\sinh^2(s)-c^2}}ds \\
v(u)&= \pm\int^u\frac{c}{\sinh(s)\sqrt{s^{2\alpha}\sinh^2(s)-c^2}}ds,
\end{aligned}
\end{equation}
where $c>0$ is a real constant.
\end{theorem}
\begin{proof}
Since $\gamma$ is not of constant curvature, both $u$ and $v$ are non-constant functions. We parametrize the curve $\gamma$ by its arc-length parameter $\sigma=\sigma(t)$ such that the functions in the appearing in the length element of $\gamma$ satisfy
\begin{equation}
u'=\cos (\sigma), \quad \sinh(u)v'= \sin(\sigma). \label{uvs-h}
\end{equation}
Taking derivative, we obtain
$$
u''=-\sigma'\sin (\sigma), \quad v''=\cos (\sigma)\frac{\sigma'\sinh(u)-\sin(\sigma)\cosh(u)}{\sinh^2(u)}.
$$
By substituting these in the expression of $\kappa$, we arrive at
$$
\kappa=\sigma'+\coth(u)\sin(\sigma).
$$
From \eqref{el-h2}, it follows
$$
\sigma'+\coth(u)\sin(\sigma)=-\alpha \frac{\sin (\sigma)}{u}.
$$
By dividing $\cos(\sigma)$, we obtain
$$
\frac{d\sigma}{du}=-\tan(\sigma)(\frac{\alpha}{u}+\coth(u)).
$$
Integrating,
$$
u^\alpha\sinh(u)\sin(\sigma)=c, \quad c\in \r,c> 0.
$$
This implies
$$
\sin(\sigma)=\frac{c}{u^\alpha\sinh(u)}, \quad \cos(\sigma)=\pm\frac{\sqrt{u^{2\alpha}\sinh^2(u)-c^2}}{u^\alpha\sinh(u)} .
$$
Considering \eqref{uvs-h}, we obtain 
$$
u'=\pm\frac{\sqrt{u^{2\alpha}\sinh^2(u)-c^2}}{u^\alpha\sinh(u)}
$$
and
$$
\frac{dv}{du}=\pm\frac{c}{\sinh(u)\sqrt{u^{2\alpha}\sinh^2(u)-c^2}}.
$$
Integrating, we may conclude \eqref{classif-h}. This completes the proof.
\end{proof}

Since the radicands appear in the integrands of \eqref{classif-h}, the variable $u$ may or not range over the entire of $(0,\infty)$. More explicitly, we require $u^{2\alpha}\sinh^2(u)-c^2>0$ or
$$
u^\alpha\sinh(u)>c, \quad u>0,c>0.
$$
Set 
$$f(u)=u^\alpha\sinh(u),\quad u\in I.$$
 Due to $u>0$, the behaviour of $f$ as $u\searrow 0$ depends on the value of $\alpha$ and hence we separate three cases:
\begin{enumerate}
\item Case $\alpha=-1$. Then, it follows that $f(u)$ is increasing and $\lim_{u\to 0}f(u)=1$. Thus, the domain of $u$ is $(0,\infty)$ when $0<c\leq1$. Otherwise, $c>1$, there is a positive constant $u_0$ such that $f(u_0)=c$, implying that the domain of $u$ is $(u_0,\infty)$. Consequently, we have
$$
I=\begin{cases}
(0,\infty), & 0<c\leq 1,\\
(u_0,\infty), & c>1,
\end{cases}
$$
where $u_0$ is the unique solution of $f(u)=c$.
\item Case $\alpha>-1$. In this case, $f(u)$ is increasing and 
$\lim_{u\to 0}f(u)=0$. Then, $I=(u_0,\infty)$ with $f(u_0)=c$.
\item Case $\alpha<-1$. We have $\lim_{u\to 0}f(u)=\infty$. Let $u_0$ denote the critical point, $f'(u_0)=0$. Then, $u_0$ solves $\alpha \sinh(u)+u\cosh(u)=0$. We conclude that $f(u)$ decreases on $(0,u_0)$ and increases on $(u_0,\infty)$. Also,
$$
I=\begin{cases}
(0,\infty), & 0<c<f(u_0),\\
(u_0,\infty)\setminus \{u_0\}, & c=f(u_0), \\
(0,a)\cup (b,\infty), & c>f(u_0),
\end{cases}
$$
where $a<u_0<b$ are two solutions of $f(u)=c$.
\end{enumerate}

\begin{remark} From the proof of Theorem \ref{t29}, it is possible to express the stationary curve equation \eqref{el-h22} as an ODE where  it only appears the function $u=u(t)$. Indeed, we have  
$$v'(t)=\frac{\sin(\sigma)}{\sinh(u)}=\frac{c}{u^\alpha\sinh(u)^2}.$$
From this identity, we can obtain $v''(t)$ and replace $v'$ and $v''$ in \eqref{k1} and \eqref{el-h22}, obtaining 
$$ u''=\frac{c^2}{u^{2 \alpha +1}\sinh(u)^2}(\alpha  -  u \coth (u) ).$$

\end{remark}
In Figure \ref{fig1}, and for different values of $\alpha$, we show some examples of $\alpha$-stationary curves by solving numerically Equation \eqref{el-h22}. We have adopted the Poincar\'e model of $\h^2$  by the symmetry of the space from the origin. The origin corresponds to the point $N=(0,0,1)$ of $\l^3$.  

\begin{figure}[h]
\begin{center}
\includegraphics[width=.3\textwidth]{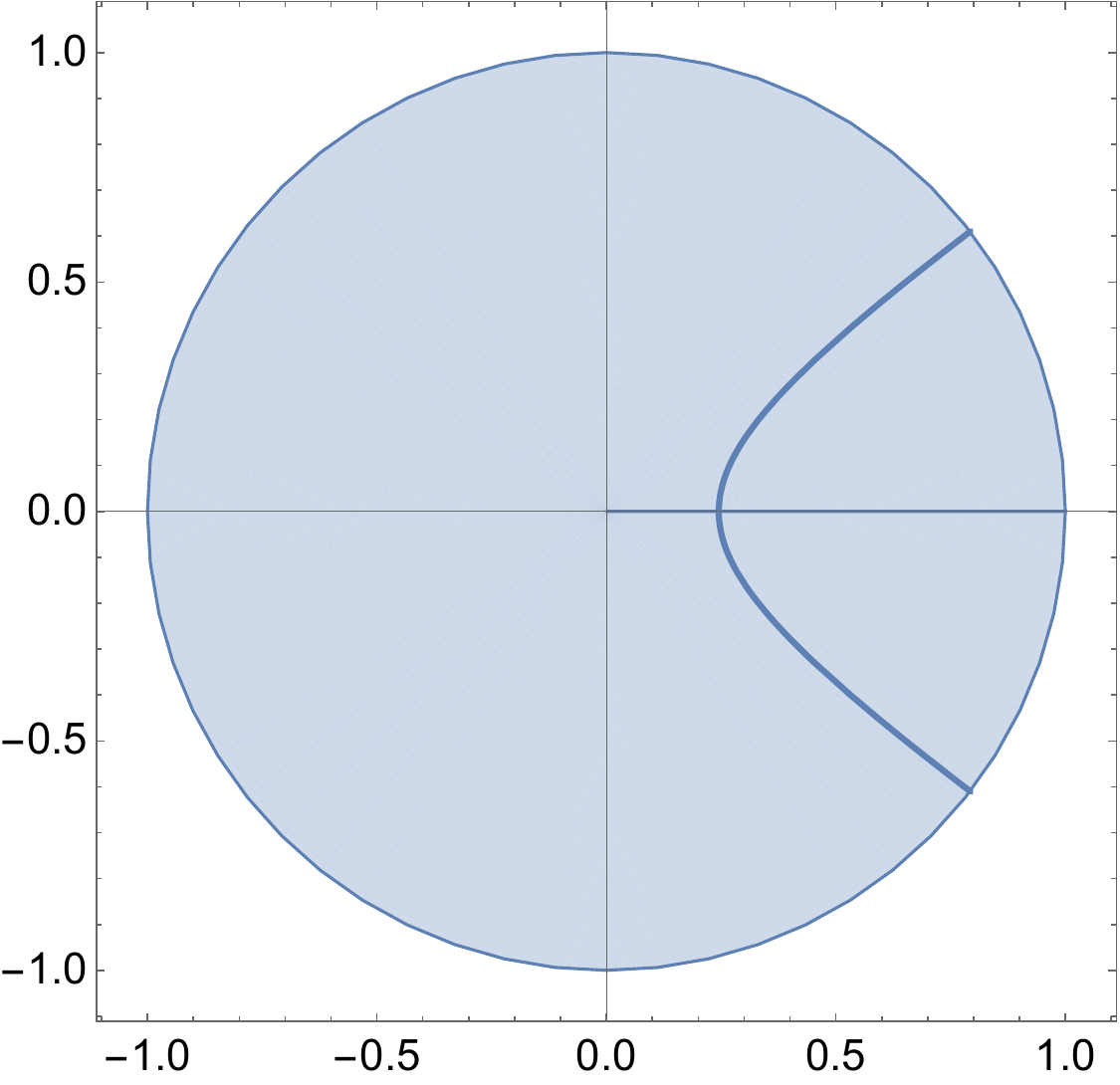} 
\includegraphics[width=.3\textwidth]{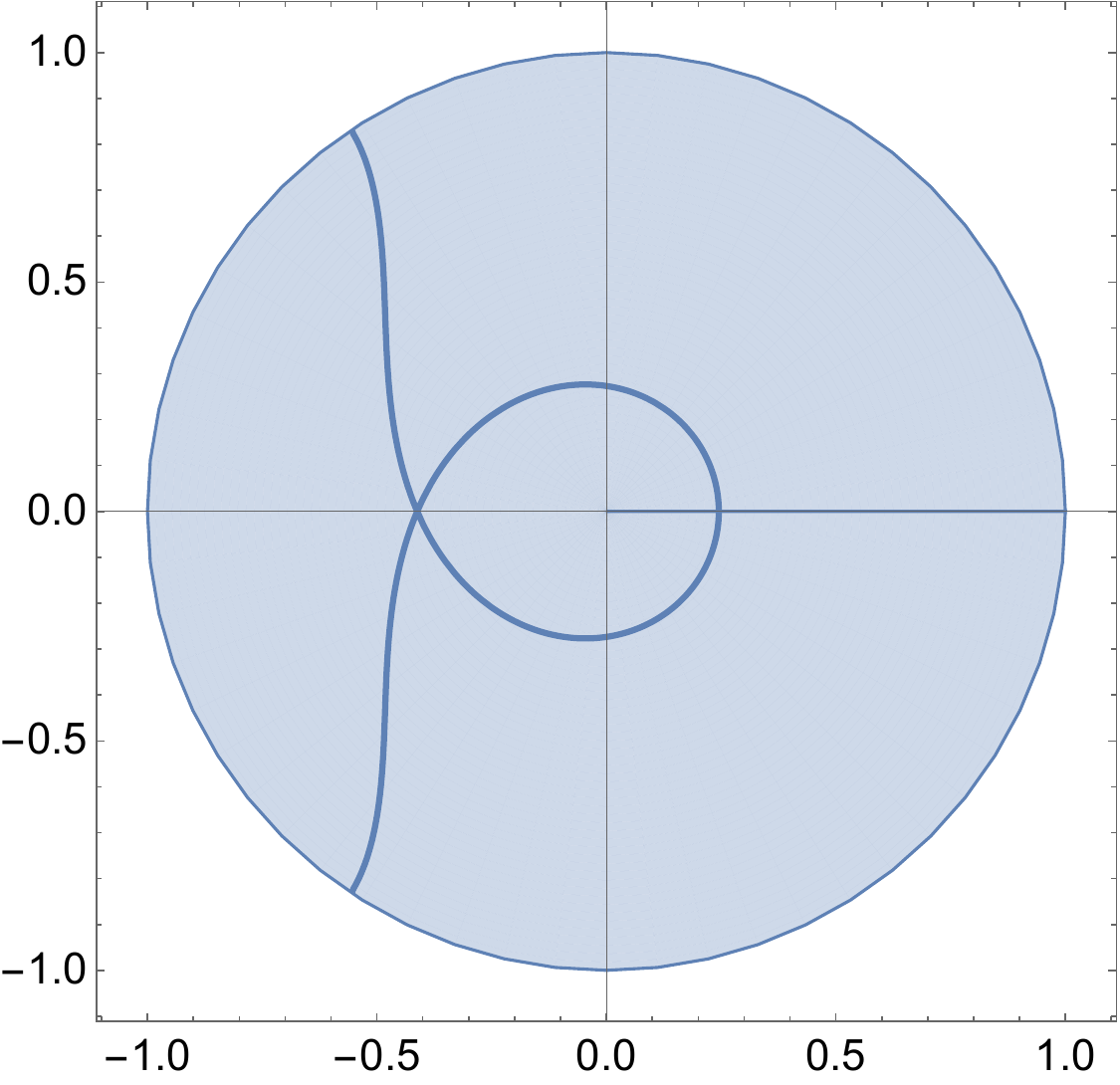}
\includegraphics[width=.33\textwidth]{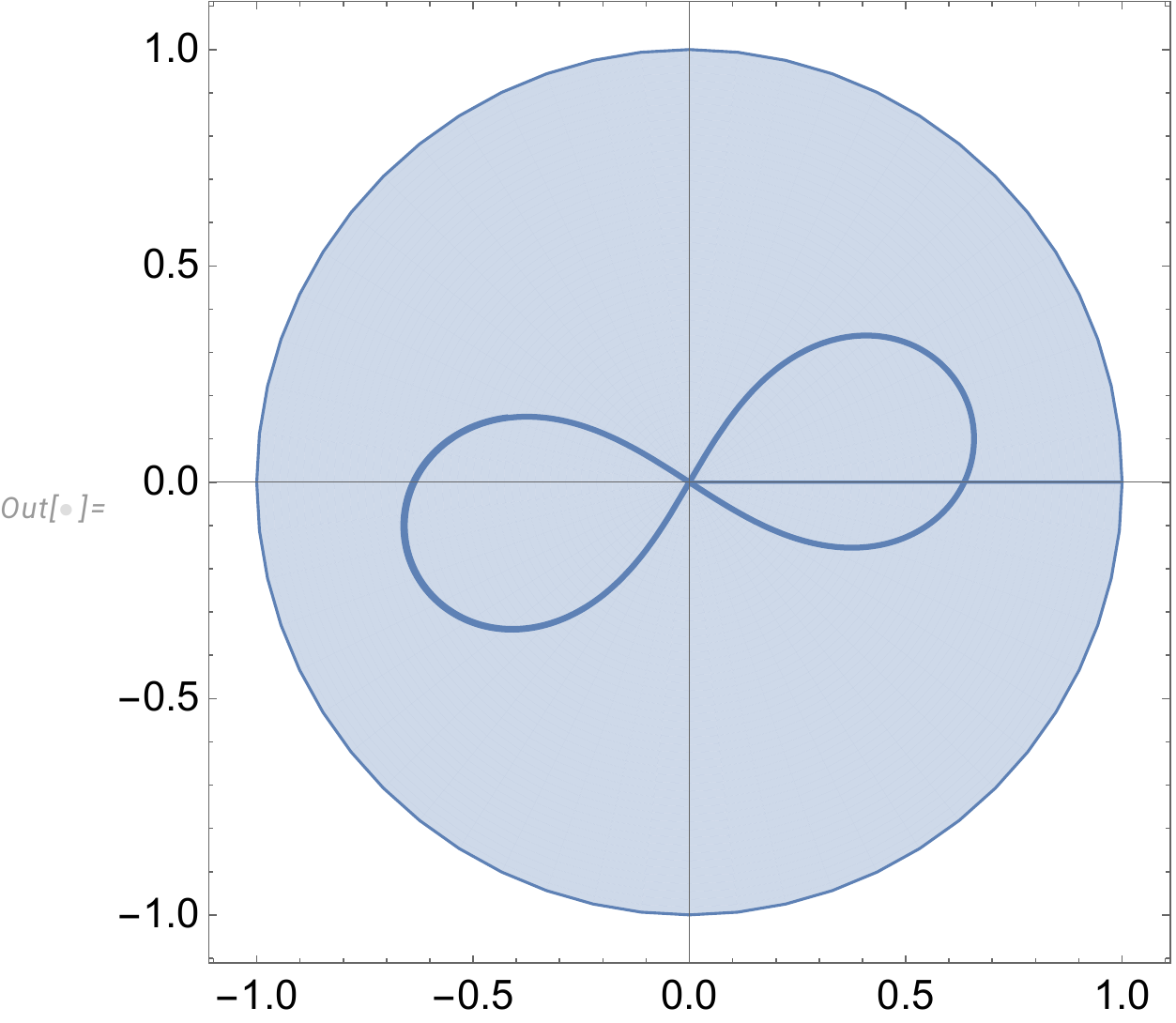}\end{center}
\caption{In the Poincar\'e model of $\h^2$, $\alpha$-stationary curves for $\alpha=1$ (left), $\alpha=-1$ (middle) and $\alpha=-3$ (right).}
\label{fig1}
\end{figure}

\subsection{Energy minimization problem in $\h^2$}

We finish this section coming back to the initial problem of finding minimizers of the energy $E_\alpha$. More clearly, given two points $p_1,p_2\in\h^2$, we find the curves $\gamma$ which join both points and globally minimize.  We will study the particular case where the two points are collinear with $N$, that is, $p_1$, $p_2$ and $N$ lie on the same geodesic. In such a case, it is expectable that this geodesic is the minimizer of $E_\alpha$. 

\begin{theorem}\label{th1}
 Let $p_1,p_2\in\h^2$ be two points lying on the same geodesic with $N$.
\begin{enumerate}
\item If $p_1$ and $p_2$ lie on the same ray starting at $N$, then the piece of this ray joining the two points is the minimizer of $E_\alpha$ for all $\alpha$.
\item Suppose $\alpha>0$. If $p_1$, $p_2$ and $N$ lie on the same geodesic and $N$ is in the middle of $p_1$ and $p_2$, then the piece of the geodesic joining the two points is the minimizer of $E_\alpha$.
\end{enumerate}
\end{theorem}

\begin{proof}
\begin{enumerate}
\item Suppose that $p_1$ and $p_2$ lie on the same ray starting at $N$. Without loss of generality, we assume that $p_1$ is closer to $N$ than $p_2$. Then there is $v_0\in\r$ such that we can parametrize the ray from $p_1$ to $p_2$ as 
$\beta(t)=\Psi(t,v_0)$ for $t\in[a_1,a_2]$ with $\beta(a_i)=p_i$, $i=1,2$. Then, we have
$$E_\alpha[\beta]=\int_{a_1}^{a_2}t^\alpha\, dt.$$
If $\gamma$ is any curve joining $p_1$ and $p_2$, then $\gamma(t)=\Psi(u(t),v(t))$ for $t\in [t_1,t_2]$ with $\gamma(t_i)=p_i$, $i=1,2$. Hence,
\begin{equation*}
\begin{split}
E_\alpha[\gamma]&=\int_{t_1}^{t_2}u(t)^\alpha\sqrt{u'(t)^2+(\sinh(u(t))v'(t))^2}\, dt\\
&\geq\int_{t_1}^{t_2}u(t)^\alpha u'(t) dt=E_\alpha[\beta].
\end{split}
\end{equation*}
\item Suppose now that $N$ is in the middle of $p_1$ and $p_2$. Then there are $v_0,a_1,a_2\in\r$   with $a_1<0<a_2$, such that the geodesic joining $p_1$ and $p_2$ can be parametrized by 
\begin{equation}\label{beta}
\beta(t)=\left\{\begin{array}{ll} \Psi(a_1-t,v_0)& t\in [ a_1,0]\\
\Psi(t,v_0+\pi)&t\in [0,a_2].
\end{array}\right.
\end{equation}
The computation of the energy of $\beta$ gives
$$E_\alpha[\beta]=\int_{a_1}^{0}(-t)^\alpha\, dt+\int_{0}^{a_2}t^\alpha\, dt.$$
On the other hand, let $\gamma=\gamma(t)$ be any curve joining $p_1$ with $p_2$,  where $t$ indicates the arc-length parameter of $\gamma$. Notice that $t$ is also the arc-length parameter of $\beta$ in \eqref{beta} because $|\Psi_u|_\epsilon=1$ as well as the distance between $\beta(t)$ and $N$. After a translation, we assume that the domain of $\gamma$ is $[a_1,b_2]$ with $a_2<b_2$, $\gamma(a_1)=p_1$ and $\gamma(b_2)=p_2$.   Notice that  the length of $\beta$ is $ a_2-a_1$ which it is less than of $\gamma$, i.e. $b_2-a_1$. 

Denote by $\d(t)$ the distance between  $\gamma(t)$ and $N$.    Since $t$ is the length parameter in both curves, then the distance between $\beta(t)$ and $N$ is less than $\d(t)$. Using $\alpha>0$, we obtain 
\begin{equation*}
\begin{split}
E_\alpha[\gamma]&>\int_{ a_1}^0\d(t)^\alpha\, dt+\int_{b_2-a_2}^{b_2}\d(t)^\alpha\, dt\geq 
 \int_{a_1}^{0}( -t)^\alpha\, dt+\int_{b_2-a_2}^{b_2}t^\alpha\, dt\\
&\geq\int_{a_1}^{ 0}( -t)^\alpha\, dt+\int_{0}^{a_2}t^\alpha\, dt= E_\alpha[\beta].
 \end{split}
\end{equation*}

\end{enumerate}
\end{proof}

\section{Sphere}

As similar to the previous section, we find, in $\s^2$, the Euler–Lagrange equation for \eqref{eq1},   examples of stationary curves, applications of the maximum principle, parametrizations of these curves, and minimization of energy. Since the computations are similar, we omit the details.

\subsection{The Euler-Lagrange equation}

We consider the   parametrization for $\s^2$ as subset of $\r^3$ given by 
$$\Psi(u,v)=(\sin (u)\cos (v),\sin (u)\sin (v),\cos (u)),\quad u,v\in\r.$$
Consider the   north pole $N=(0,0,1)\in\s^2$ as the reference point for the distance function. Let $\gamma\colon [a,b]\to\s^2$ be a curve, $\gamma=\gamma(t)$, given by 
$$\gamma(t)=\Psi(u(t),v(t))=(\sin u(t)\cos v(t),\sin u(t)\sin v(t),\cos u(t)),$$
where $u=u(t)$, $v=v(t)$ are smooth functions on $[a,b]$. Then, the distance from $\gamma(t)$ to $N$ is $\d(t)=u(t)$. Since the line element is given by
$$|\gamma'|=\sqrt{u'^2+\sin(u)^2 v'^2},$$
the energy \eqref{eq1} becomes  
\begin{equation} \label{en-s}
E_\alpha[\gamma]=\int_a^b u^\alpha\sqrt{u'^2+\sin^2 (u) v'^2}\, dt.
\end{equation}
The normal is defined by $\n(t)=\frac{\gamma'(t)\times\gamma(t)}{|\gamma'(t)|}$, obtaining
$$
\n=\frac{\gamma'\times\gamma}{|\gamma'|}
=\frac{1}{\sqrt{u'^2+\sin^2(u)v'^2}}
\begin{pmatrix}
u'\sin v+v'\sin (u)\cos (u)\cos (v)\\
-u'\cos (v)+v'\sin (u)\cos (u)\sin (v)\\
-v'\sin^2(u)
\end{pmatrix}.
$$
The   curvature $\kappa$ of $\gamma(t)$ is 
\begin{equation}\label{k2}
\begin{split}
\kappa&=\frac{\langle\gamma'',\n\rangle}{|\gamma'|^3}=\frac{\mbox{det}(\gamma'',\gamma',\gamma) }{|\gamma'|^3}\\
&=-\frac{1}{(u'^2+\sin ^2(u) v'^2)^{3/2}}\left(v'\cos (u)(2u'^2+v'^2\sin^2(u))+\sin (u)(u'v''-u''v')\right).
\end{split}
\end{equation}

The Euler-Lagrange equations for \eqref{en-s} are
\begin{equation*} 
\begin{split}
0&= u^{\alpha-1}v'\sin (u) \Big( \alpha v'\sin (u)(u'^2+\sin^2  (u)v'^2)\\
&+u(\sin (u)(u'v''-u''v') +v'\cos (u)(2u'^2+v'^2\sin(u)^2))\Big),\\
0&= u^{\alpha-1}u'\sin (u) \Big( \alpha v'\sin (u)(u'^2+\sin^2  (u)v'^2)\\
&+u(\sin (u)(u'v''-u''v') +v'\cos (u)(2u'^2+v'^2\sin(u)^2))\Big).
\end{split}
\end{equation*}
By regularity of $\gamma$, the functions $u'$ and $v'$ cannot vanish simultaneously. Thus the parenthesis in the above two equations is $0$. By using the expression of $\kappa$ given in \eqref{k2}, we get 
$$
\kappa=\alpha \frac{v'\sin (u)}{u|\gamma'|}.
$$
As in Section \ref{sec-h}, we may conclude the following characterization of stationary curves in $\s^2$.

\begin{proposition}\label{s-p1}
The $\alpha$-stationary curves $\gamma$ in $\s^2$ are characterized in terms of their curvature $\kappa$ by 
\begin{equation} \label{el-s1}
\kappa=\alpha\frac{\langle \n,\xi\rangle}{\d }.
\end{equation}
Here, $\n$ denotes the unit normal vector  of $\gamma$, $\d$ is the distance from $N$, and $\xi$ is the unitary tangent vector  to the minimizing geodesic joining $\gamma(t)$ and $N$. 
\end{proposition}
   
It is necessary to point out that in $\s^2$, the notion of ray holds as in $\h^2$ in the sense that it is the geodesic from $N$ to a point of $\s^2$ with the extra condition that this geodesic is minimizing the length. Notice that given a point $p\in\s^2$ there are two arcs of geodesics joining $N$ with $p$, but only one (ray) is minimizing the length (except that $p=-N$, where all geodesics are minimizers of the length). Again, as in the cases in $\h^2$ and $\r^2$,  the characterization of $\alpha$-stationary curve given by \eqref{el-s1} has the same form as \eqref{el-h2} and \eqref{eq00}, respectively.  

It is also clear that  rotations about the $z$-axis and reflections about planes containing the $z$-axis preserve the solutions of \eqref{el-s1}.
 
\subsection{Examples of stationary curves}

The following are immediate examples of stationary curves:
\begin{enumerate}
\item Geodesics crossing $N$ are $\alpha$-stationary curves for all $\alpha$.
\item Circles centered at $N$. A circle of radius $r>0$ centered at $N$ is parametrized by 
$$\gamma(t)=(\sin (r) \cos (t),\sin (r) \sin (t),\cos (r)).$$
The inward normal is $\n(t)=(-\cos (r) \cos (t),-\cos (r) \sin (t),\sin (r))$. Hence, $\langle\n,\xi\rangle=-1$ and $\kappa=\cot(r)$, $r\in (0,\pi) $. Thus $\gamma$ is an $\alpha$-stationary curve for $\alpha = -r\cot(r)$.
\end{enumerate}

 It is worth pointing out that, for $r\in (0,\pi) $, we have $\alpha = -r\cot(r) \in(-1,\infty)$, while in $\h^2$ the value of $\alpha$ is always negative.

Again, the next objective is finding all stationary curves in $\s^2$ with constant curves. The description of the curves of $\s^2$ with constant curvature is the following.  

\begin{proposition} \label{p-tau-s} 
The curves in $\s^2$ with constant curvature are described as follow. Let 
$$C_{a,\tau}=\{p\in\s^2\colon\langle p,a\rangle=\tau\}.$$
The normal is
$$\n(p)=\lambda( a-\tau p),\quad\lambda=\frac{1}{\sqrt{1-\tau^2}},$$
and the curvature is $\kappa=\lambda\tau$. The types are the following:
\begin{enumerate}
\item Geodesics. Here $\tau=0$ and $\kappa=0$. They are great circles of $\s^2$.
\item Circles. Here $0<|\tau|<1$ and $\kappa=\frac{\tau}{\sqrt{1-\tau^2}}$.
\end{enumerate}
\end{proposition}

The classification of stationary curves in $\s^2$ with constant curvature is the following. The proof is analogue to Theorem \ref{h-ccurv} and we omit it.

\begin{proposition} \label{s-ccurv}
The only   $\alpha$-stationary curves in $\s^2$ with constant curvature are: 
\begin{enumerate}
\item Geodesics passing through $N$. This holds for all value of $\alpha$.
\item Circles of radius $r$ centered at $N$ for $\alpha=-r\cot (r)$. 
\end{enumerate}
\end{proposition}

\subsection{The maximum principle}

As in the hyperbolic plane, we study the stationary curves in $\s^2$ which also are  closed curves. A key difference with the hyperbolic plane is that the ambient space $\s^2$ is compact. Thus all curves are bounded and its distance from $N$ is less than $\pi$. For example, the value of the radius $r$ of the circles is not arbitrary because $r\in (0,\pi)$. The expression of the weighted curvature $\kappa^\phi$ coincides with \eqref{kk}, where the metric $\langle , \rangle_\epsilon$ is now replaced by $\langle , \rangle$.

 Denote by $\s^2_+=\s^2\cap\{z>0\}$ the upper hemisphere and by $\s^2_-=\s^2\cap\{z<0\}$ the lower hemisphere.
\begin{theorem} 
Let $\gamma$ be an $\alpha$-stationary closed curve in $\s^2$. 
\begin{enumerate}
\item If $\gamma$ is contained in the open hemisphere $\s^2_+$, then $\alpha<0$.
\item If $\gamma$ is contained in the open hemisphere $\s^2_-$, then $\alpha>0$.
\end{enumerate}
\end{theorem}

\begin{proof}
\begin{enumerate}
\item Suppose $\gamma(I)\subset\s^2_+$. Let  $C_{r}$ be a circle centered at $N$ of radius $r>0$. For $r$ close to $\pi/2$, the curve $\gamma$ is contained in the disc $D_r$ determined by $C_r$ and including $N$. If $ r\searrow 0$, let $r_1>0$ be the radius of the first circle that touches $\gamma$. With the orientation on $C_{r_1}$ pointing to $D_{r_1}$, we have $\gamma\geq C_{r_1}$ around the contact point. Since the curvature of $C_{r_1}$ is $r_1\cot(r_1)$ and the normal vector on $C_{r_1}$ is the opposite of $\xi$,  the weighted curvature $\kappa^\phi$ for the value $\alpha$ of $C_{r_1}$ is 
$$\kappa^\phi_{C_{r_1}}=\cot(r_1)+ \frac{\alpha}{r_1}=\frac{\alpha+r_1\cot(r_1)}{r_1}.$$
Then the maximum principle implies  
$$0=\kappa^\phi_\gamma\geq  \kappa^\phi_{C_{r_1}} =\frac{\alpha+r_1\cot(r_1)}{r_1}.$$
This gives $\alpha\leq -r_1\cot(r_1)$.   Since $r_1\in (0,\pi/2)$, then $\alpha=-r_1\cot(r_1)< 0$, proving the first item. 
\item Now suppose $\gamma(I)\subset\s^2_-$. Consider a circle $C_r$ centered at $N$ with radius $r>\pi/2$. For $r$ close to $\pi/2$, $\gamma$ lies outside the disc $D_r$ bounded by $\gamma$ and containing $N$. Let $r\nearrow\pi$ until the first contact with $\gamma$ for some radius $r_2$. On $C_{r_2}$ consider the orientation pointing outside $D_{r_2}$ and thus $\kappa=-\cot(r_2)$. The normal vector on $C_{r_2}$ coincides with $\xi$, implying 
$$\kappa^\phi_{C_{r_2}}=-\cot(r_2)- \frac{\alpha}{r_2}=\frac{ -r_2\cot(r_2)-\alpha}{r_2}.$$
Since $\gamma\geq C_{r_2}$  around the contact point, and because $\kappa^\phi_\gamma=0$, the maximum principle gives 
$$0\geq \frac{ -r_2\cot(r_2)-\alpha}{r_2}.$$
Thus $\alpha\geq -r_2\cot(r_2)>0$, because $r_2\in (\frac{\pi}{2},\pi)$. This proves the second item.
\end{enumerate}
 
\end{proof}

If an $\alpha$-stationary curve is far away from $N$ but intersects $\s^2_+$, then the value of $\alpha$ can be estimated.  

\begin{theorem} \label{t12}
Let $\gamma$ be an $\alpha$-stationary curve in $\s^2$ and suppose that $\gamma$ is properly immersed. If  the north pole $N$ is not an adherent point of $\gamma(I)$ and $\gamma(I)\cap\s^2_+\not=\emptyset$, then $\alpha> -1$.
\end{theorem}

\begin{proof}
Since $N\not\in\overline{\gamma(I)}$ and $\gamma$ is properly immersed in $\s^2$, the distance between $\gamma(I)$ and $N$ is positive. For $r>0$ sufficiently small, let $C_r$ be a circle of radius $r$ and centered at $N$ such that the domain $D_r$ bounded by $C_r$ and containing $N$ does not intersect $\gamma$. Letting $r\nearrow\frac{\pi}{2}$, and because $\gamma$ is properly immersed, we arrive until the first circle $C_{r_1}$ which touches $\gamma$, where the contact occurs tangentially.  Then $r_1\in (0,\frac{\pi}{2})$ because  $\gamma(I)\cap\s^2_+\not=\emptyset$. 

Consider on $C_{r_1}$ the outward orientation. Then the normal vector of  $C_{r_1}$ coincides with $\xi$ and the curvature of $C_{r_1}$ is $-\cot(r_1)$. Thus the computation of the weighted curvature $\kappa^\phi$ of $C_{r_1}$ for the value of $\alpha$ is 
$$\kappa^\phi_{C_{r_1}}=-\cot(r_1)-\frac{\alpha}{r_1}=-\frac{r_1\cot(r_1)}{r_1}.$$
On the other hand, $\kappa^\phi_\gamma=0$ because $\gamma$ is an $\alpha$-stationary curve, regardless the orientation on $\gamma$. Since $\gamma\geq C_{r_1}$ around the contact point, the maximum principle implies $\kappa^\phi_\gamma\geq \kappa^\phi_{C_{r_1}}$, that is
$$0\geq -\frac{r_1\cot(r_1)}{r_1}.$$
Therefore $\alpha\geq r_1\cot(r_1)$. Since $r_1\in (0,\frac{\pi}{2})$, then $-r_1\cot(r_1)\in (-1,0)$. This proves the result.
 
\end{proof}

  This result is analogous to that of Euclidean plane, where the same conclusion holds if $\alpha<-1$ \cite{dl1}. 
   
\subsection{Parametrizations of stationary curves}
Since the $\alpha$-stationary curves in $\s^2$ with constant curvature are already described in Proposition \ref{s-ccurv}, we now establish the parametrizations of those with non-constant curvature. The proof is similar as that of Theorem \ref{t29} and hence we omit it.
\begin{theorem} \label{t36}
For the $\alpha$-stationary curves $\gamma(t)=\Psi(u(t),v(t))$ in $\s^2$ which are not of constant curvature we have 
\begin{equation} \label{classif-s}
\begin{aligned}
t(u)&=\pm\int^u\frac{s^\alpha\sin(s)}{\sqrt{s^{2\alpha}\sin^2(s)-c^2}}ds \\
v(u)&= \pm\int^u\frac{c}{\sin(s)\sqrt{s^{2\alpha}\sin^2(s)-c^2}}ds,
\end{aligned}
\end{equation}
where $c>0$ is a real constant.
\end{theorem}

We determine the admissible intervals of $u$ for the integrals in \eqref{classif-s}. Set 
$$f(u)=u^\alpha\sin(u),\quad I=\{u\in (0,\pi):f(u)>c\}.$$
 Also, let $f(u_0)=c$ with $u_0\in I$.
\begin{enumerate}
\item Case $\alpha=-1$.
$$
I=\begin{cases}
(0,u_0), & 0<c< 1,\\
\emptyset, & c\geq1,
\end{cases}
$$
where $u_0$ is the unique solution of $f(u)=c$.
\item Case $\alpha>-1$. 
$$
I=\begin{cases}
[u_0^1,u_0^2], & 0<c\leq f(\bar{u}),\\
\emptyset, & c>f(\bar{u}),
\end{cases}
$$
where $f(u_0^i)=c$ and $0<u_0^1<\bar{u}<u_0^2<\pi$ such that $\bar{u}$ the critical point of $f(u)$.
\item Case $\alpha<-1$. $I=(0,u_0)$.
\end{enumerate}

Figures of $\alpha$-stationary curves in $\s^2$ for different values of $\alpha$ are shown in Figure \ref{fig2}.

\begin{figure}[h]
\begin{center}
\includegraphics[width=.3\textwidth]{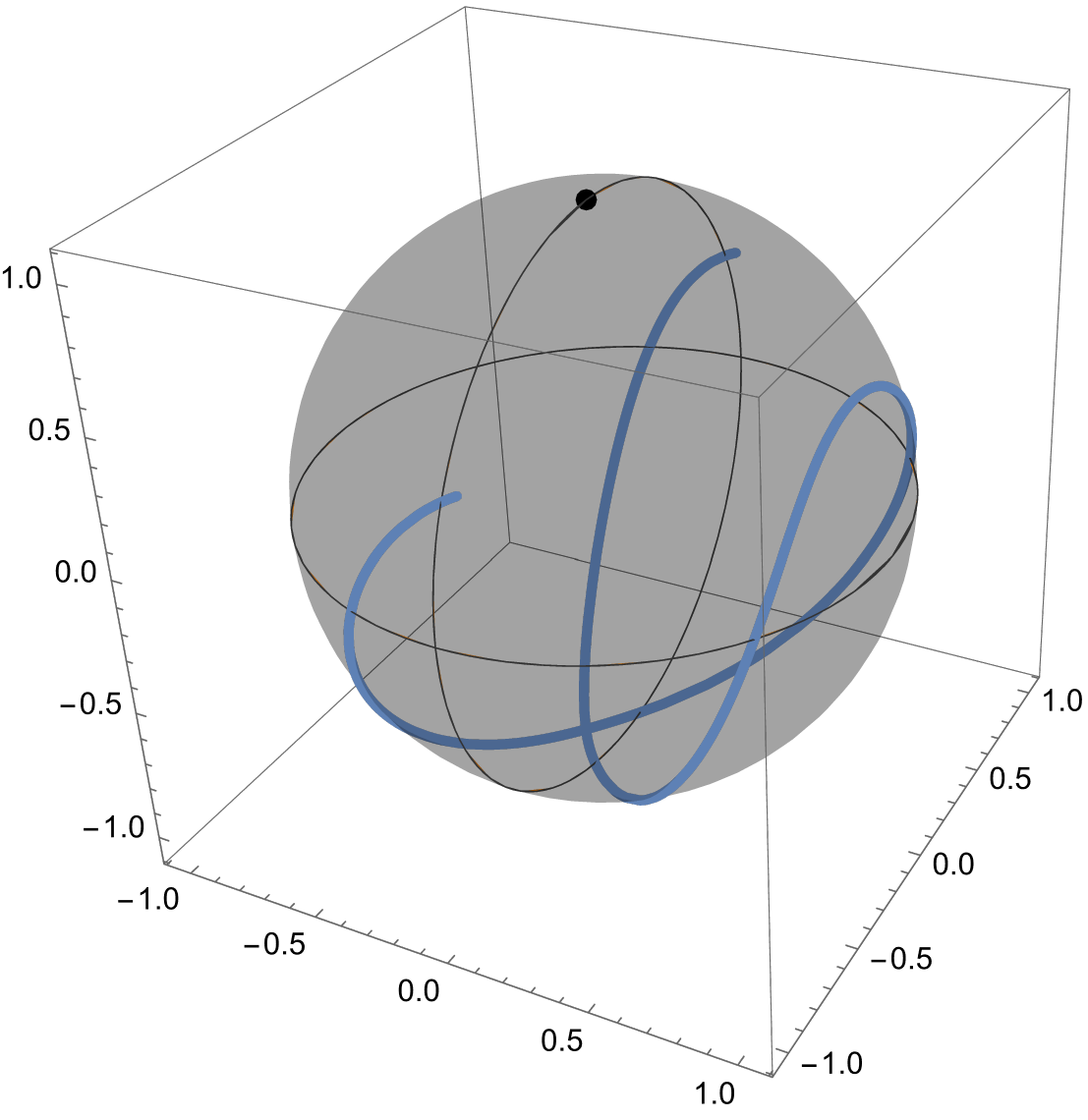} 
\includegraphics[width=.3\textwidth]{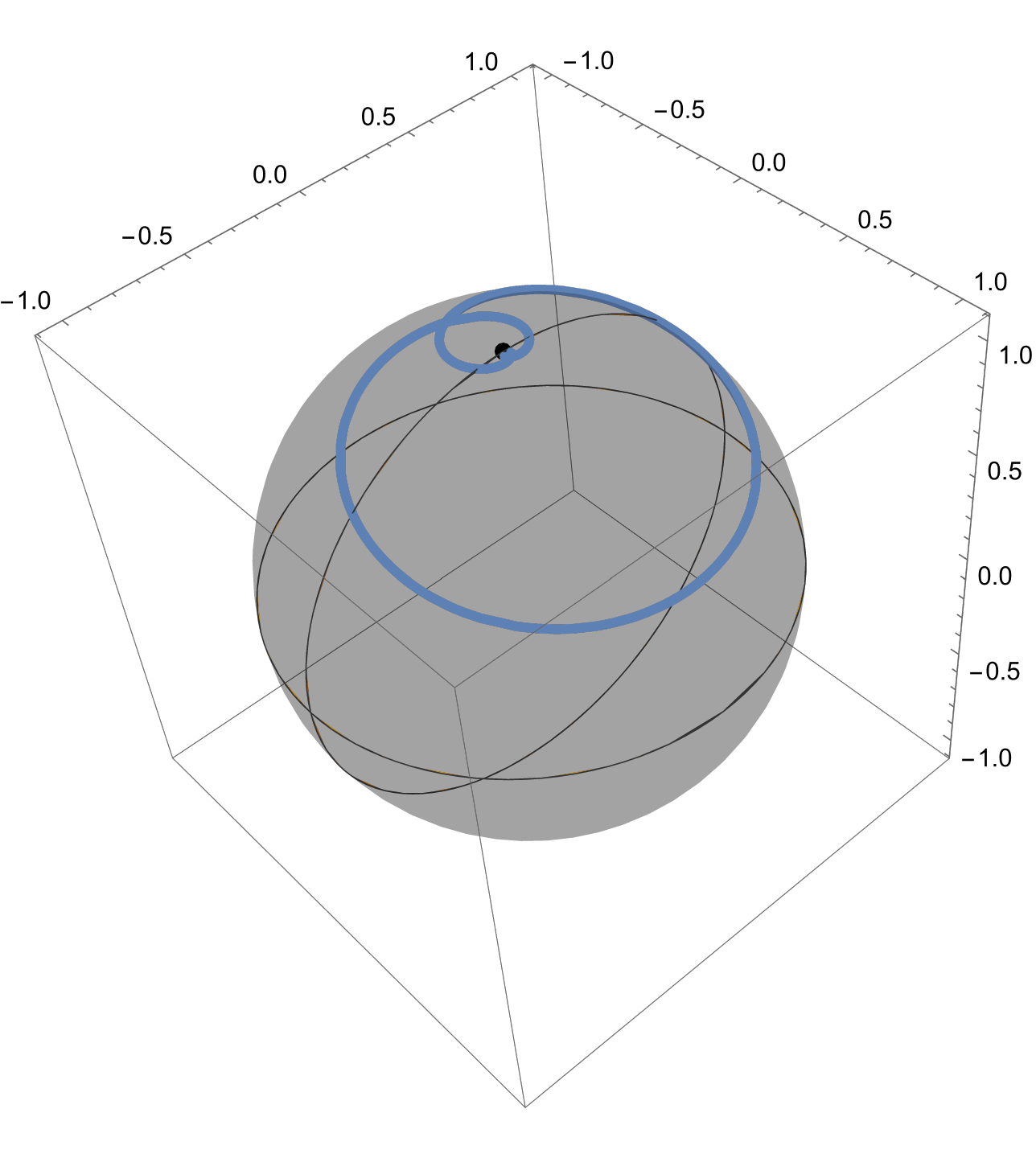}
\includegraphics[width=.3\textwidth]{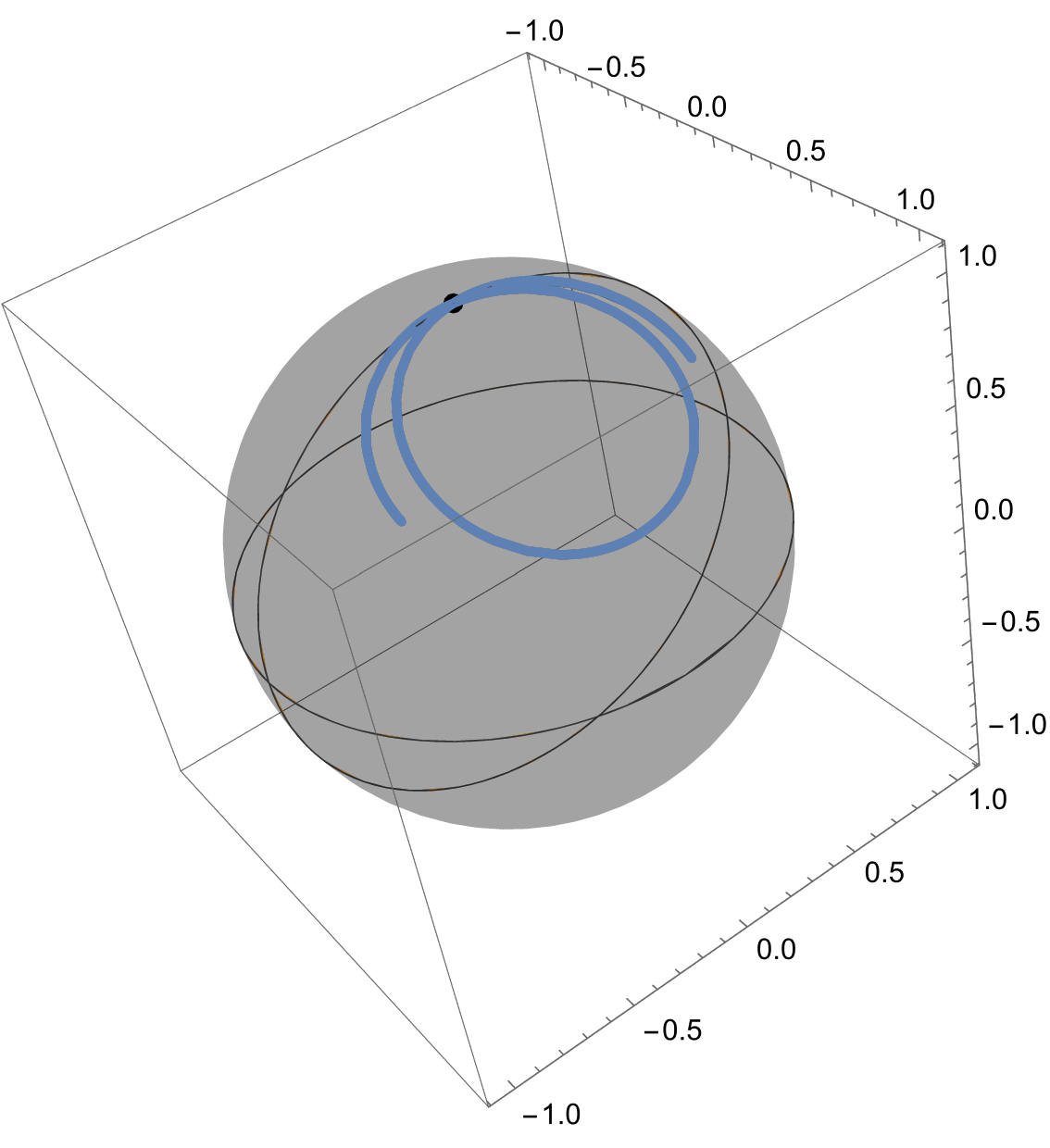}
\end{center}
\caption{Examples of $\alpha$-stationary curves in $\s^2$ for   $\alpha=2$ (left), $\alpha=-1$ (middle) and $\alpha=-2$ (rigth).}
\label{fig2}
\end{figure}

\subsection{Energy minimization problem in $\s^2$}
We address the problem of finding  minimizers     of the energy $E_\alpha$ between two given points. As in the previous section, we only consider the case that  $p_1,p_2\in\s^2$ lie on the same geodesic with  $N$. In the following result, we understand a ray starting at $N$ as a minimizing  (for the length) geodesic starting at $N$. 

\begin{theorem} \label{ts1}
Let $p_1,p_2\in\s^2$ be two points lying on the same geodesic with $N$. 
\begin{enumerate}
\item If $p_1$ and $p_2$ lie on the same ray starting at $N$, then the piece of this ray joining the two points is the minimizer of $E_\alpha$ for all $\alpha$. 
\item Suppose $\alpha>0$. If the minimizing geodesic from $p_1$ to $p_2$ contains  $N$, then this geodesic is the minimizer of $E_\alpha$. 
\end{enumerate}
\end{theorem}

\begin{proof}
\begin{enumerate}
\item The proof is analogous to the item (1) of Theorem \ref{th1}.
\item Let $\beta$ be the minimizing (for the length) geodesic joining $p_1$ and $p_2$. The proof is analogous to the item (2) of Theorem \ref{th1} because $\beta$ is a minimizing geodesic, which it is used in the proof.

\end{enumerate}
\end{proof}

\begin{remark} A case not covered in Theorem \ref{ts1} is when  $p_1$ and $p_2$ are not on the same ray and  the minimizing geodesic from $p_1$ to $p_2$ does not pass through $N$. In this case, the geodesic must pass through the south pole $(0,0,-1)$. Such a geodesic can be parametrized by $\sigma\colon [a,\pi]\cup [b,\pi]\to\s^2$, where
\begin{equation*}
\sigma(t)=\left\{\begin{array}{ll} \Psi( t,v_0)& t\in [ a,\pi]\\
\Psi(-t+\pi+b,v_0+\pi)&t\in [b,\pi]
\end{array}\right.
\end{equation*}
with the condition $\pi\leq a+b$. If $\alpha>0$, the energy of $\sigma$ is  $E_\alpha[\sigma]=\frac{1}{\alpha+1}(2\pi^{\alpha+1}-a^{\alpha+1}-b^{\alpha+1})$. Given any curve $\gamma$ joining $p_1$ and $p_2$, its length is greater than that of $\sigma$. However, after parametrizing $\gamma$ by arc-length, when moving the parameter from $p_1$ to $(0,0,-1)$ the distance between $\sigma(t)$ and $N$ increases and we can no longer estimate it in terms of the distance between $\gamma(t)$ and $N$.
\end{remark}

\subsection*{Ethics declarations}

Conflict of interest. The authors have no conflict of interest to declare that are relevant to the content of this article. No data were used to support this study

\subsection*{Acknowledgment}
Rafael L\'opez  has been partially supported by MINECO/MICINN/FEDER grant no. PID2023-150727NB-I00,  and by the ``Mar\'{\i}a de Maeztu'' Excellence Unit IMAG, reference CEX2020-001105- M, funded by MCINN/AEI/10.13039/ 501100011033/ CEX2020-001105-M.


\end{document}